\documentclass[12pt, twoside, fleqn, a4paper]{article}

\usepackage{latexsym}  
\usepackage{amscd}    
\usepackage{theorem}  
\usepackage{pifont}  
\usepackage{mathbbol} 
\usepackage{amsfonts}  
\usepackage{xspace}  
\usepackage{amssymb} 
\usepackage{fancyhdr} 
\usepackage{mathrsfs} 
\usepackage{amsmath} 
\usepackage{graphics} 
\usepackage{graphicx} 
\usepackage{euscript}
\usepackage{psfrag}
\usepackage{empheq} 
\usepackage{mathabx} 
\usepackage[parfill]{parskip}     
\usepackage[colorlinks=true, allcolors=blue]{hyperref}

\title{A variational principle for holomorphic correspondences}
\author{Subith Gopinathan\footnote{Indian Institute of Science Education and Research Thiruvananthapuram (IISER-TVM), Maruthamala P.O., Vithura, Kerala, India. PIN 695 551.\ \ email: \texttt{subith21@iisertvm.ac.in}}\ \ and\ \ Shrihari Sridharan\footnote{Indian Institute of Science Education and Research Thiruvananthapuram (IISER-TVM), Maruthamala P.O., Vithura, Kerala, India. PIN 695 551.\ \ email: \texttt{shrihari@iisertvm.ac.in} (Corresponding Author). \\ The second named author thanks the support provided by NBHM through a research grant \\ No. 02011/35/2025/NBHM(R.P)/R\&D II/9832}} 

\DeclareFontFamily{OT1}{pzc}{}
\DeclareFontShape{OT1}{pzc}{m}{it}%
              {<-> s * [0.900] pzcmi7t}{}
\DeclareMathAlphabet{\mathpzc}{OT1}{pzc}%
                                 {m}{it}

{\theorembodyfont{\slshape} \newtheorem{theorem}{Theorem}[section]}
{\theorembodyfont{\slshape} \newtheorem{definition}[theorem]{Definition}}
{\theorembodyfont{\slshape} \newtheorem{lemma}[theorem]{Lemma}}
{\theorembodyfont{\slshape} }
{\theorembodyfont{\slshape} }
{\theorembodyfont{\slshape} \newtheorem{corollary}[theorem]{Corollary}} 
{\theorembodyfont{\slshape} } 
{\theorembodyfont{\slshape} }
{\theorembodyfont{\slshape} }

\numberwithin{equation}{section}

\newenvironment{proof}{\paragraph{Proof:}}{\hfill$\bullet$}

\date{March 04, 2026}

\begin{document}

\maketitle 

\begin{abstract} 
In this paper, we consider a dynamical system on the Riemann sphere that evolves through a set-valued map, namely a holomorphic correspondence. Analogous to the investigation of the dynamics effected by a continuous map defined on a compact metric space, wherein the concept of measure-theoretic entropy of the map and its utility in defining the pressure of a function are well-studied, we define the measure-theoretic entropy of a holomorphic correspondence and use the same to define the pressure of continuous functions. These ideas naturally lead to the formulation of a variational principle in the context of the dynamics of a holomorphic correspondence. 
\end{abstract}

\begin{tabular}{l l} 
{\bf Keywords} : & Holomorphic correspondence, \\
& Measure theoretic entropy, \\
& Variational principle, \\
& Ruelle operator. \\ 
& \\ 
{\bf MSC Subject} & 37D35, 37A35, 37F05. \\ 
{\bf Classifications} & \\ 
\end{tabular} 
\bigskip 

\newpage 

\section{Introduction}

The variational principle for maps, see \cite{Pw:1975}, states thus: For a continuous map $T$ defined on a compact metric space $X$, the pressure functional for any function $f$ in the Banach space of continuous real-valued functions defined on $X$, with respect to $T$, is given by 
\begin{equation} 
\label{eqn:varprT} 
{\rm Pr}_{T} (f) = \sup_{m\, \in\, \mathscr{M}_{T} (X)} \left\{ h_{T} (m) + \int_{X} f \mathrm{d}m \right\}, 
\end{equation} 
where $\mathscr{M}_{T} (X)$ is the space of all $T$-invariant probability measures supported on $X$ and $h_{T} (m)$ is the measure-theoretic entropy of the measure $m$ corresponding to the map $T$. 

In this paper, we are interested in expression of the pressure functional, analogous to the one stated in Equation \eqref{eqn:varprT} for any function $f$ in the Banach space of continuous real-valued functions defined on the Riemann sphere (denoted by $\widehat{\mathbb{C}}$), nevertheless, with respect to a holomorphic correspondence, say $\Gamma$ that we now define. 

\begin{definition} 
\label{holo}
A holomorphic correspondence $\Gamma \subseteq \widehat{\mathbb{C}} \times \widehat{\mathbb{C}}$ is given by a formal linear combination of the form $\displaystyle{\Gamma = \sum_{t\, =\, 1}^{M} \Gamma_{t}}$, where $\Gamma_{t}$’s are irreducible complex-analytic subvarieties, repeated according to their multiplicities that satisfy the following conditions: Suppose $\pi_{i}$ denotes the projection onto the $i$-th coordinate for $i = 1, 2$, then 
\begin{enumerate} 
\item $\pi_{i} \vert_{\Gamma_{t}}$ is surjective for each $i = 1, 2$ and for every $1 \le t \le M$ and 
\item the sets $\pi_{1} \left( \pi_{2}^{-1} \left( \left\{ x_{0} \right\} \right) \cap \Gamma_{t} \right)$ and $\pi_{2} \left( \pi_{1}^{-1} \left( \left\{ x_{0} \right\} \right) \cap \Gamma_{t} \right)$ are of finite cardinality for any generic point $x_{0} \in \widehat{\mathbb{C}}$ and for every $1 \le t \le M$. 
\end{enumerate} 
\end{definition} 

Throughout this paper, we use the term ``set", as in the second condition of Definition \ref{holo}, to represent the collection of points in $\widehat{\mathbb{C}}$ so defined, counted with multiplicities. Suppose for some $x_{0} \in \widehat{\mathbb{C}}$, we denote $d_{{\rm top}} \left( x_{0} \right) = \# \left( \pi_{1} \left( \pi_{2}^{-1} \left( \left\{ x_{0} \right\} \right)\cap \Gamma_{t} \right) \right)$. We know that there exists a non-empty Zariski open set, say $U \subseteq \widehat{\mathbb{C}}$ such that, for $1 \le t \le M$, defining $U^{t} := \pi_{2}^{-1} (U) \cap \Gamma_{t}$, the map $\left. \pi_{2} \right|_{U_{t}}$ provides a $\delta_{t}$-sheeted holomorphic covering, where $\delta_{t} = \# \left\{ y \in \widehat{\mathbb{C}}\ :\ (y,\, x_{0}) \in \Gamma_{t} \right\}$ for any $x_{0} \in U$. Hence, $d_{{\rm top}} = \sum\limits_{t\, =\, 1}^{M} \delta_{t}$, for any $x_{0} \in U$. Similarly, one defines $d_{{\rm fwd}}$ for the holomorphic correspondence $\Gamma$ by considering the cardinality of the set $\pi_{2} \left( \pi_{1}^{-1} \left( \left\{ x_{0} \right\} \right)\cap \Gamma_{t} \right)$, defining the quantity $\lambda_{t} = \# \left\{ y \in \widehat{\mathbb{C}}\ :\ (x_{0},\, y) \in \Gamma_{t} \right\}$ for $x_{0} \in \widehat{\mathbb{C}}$ and obtaining $d_{{\rm fwd}} = \sum\limits_{t\, =\, 1}^{M} \lambda_{t}$. We urge the readers to note that the case where $d_{{\rm top}} = d_{{\rm fwd}} = 1$ for all points $x_{0} \in \widehat{\mathbb{C}}$ reduces $\Gamma$ to the graph of a M\"{o}bius map, while $d_{{\rm top}} \ge 2$ and $d_{{\rm fwd}} = 1$ is the study of rational maps of degree $d_{{\rm top}}$. 

For any $f \in \mathcal{C} \left( \widehat{\mathbb{C}}, \mathbb{R} \right)$, we aim to define the pressure functional of $f$ with respect to the holomorphic correspondence $\Gamma$ as 
\[ {\rm Pr}_{\Gamma} (f) = \sup_{m\, \in\, S_{\Gamma}\, \subseteq\, \mathscr{M} \left( \widehat{\mathbb{C}} \right)} \left\{ \gamma (m) + \int_{\widehat{\mathbb{C}}} f \mathrm{d}m \right\}, \] 
for an appropriate set $S_{\Gamma}$ of measures dependent on the correspondence that is a subset of the space of all probability measures supported on the Riemann sphere, denoted by $\mathscr{M} \left( \widehat{\mathbb{C}} \right)$. Also, one needs to identify the function $\gamma : \mathscr{M} \left( \widehat{\mathbb{C}} \right) \longrightarrow \mathbb{R}$, that can act analogous to the measure-theoretic entropy in the case of maps. 

The study of thermodynamical properties of holomorphic correspondences has gained significant attention in recent times. In \cite{KS:2022}, the authors introduced the notion of a metric entropy for an invariant measure associated to the set-valued map, say $F : X \longrightarrow 2^{X}$, where $2^{X}$ denotes the collection of all subsets of $X$. Further, they proved that the supremum of all these metric entropies is atmost the topological entropy of $F$, proving what they call the \emph{half variational principle}. In \cite{LLZ:2024}, the authors prove a variational principle for a certain class of correspondences by considering the space of continuous functions defined on the graph of the correspondence satisfying certain regularity conditions and using what they call \emph{transition probability kernels}. In \cite{ds:2008}, the authors introduced the notion of entropy of holomorphic correspondences using the idea of separated sets and obtain a fair estimate for the upper bound for the entropy of holomorphic correspondences. In \cite{bs:2021}, the authors make use of this upper bound and obtain a class of holomorphic correspondences generated by a rational semi-group for which the upper bound is achieved. In \cite{ss:2025}, the authors defined the entropy of holomorphic correspondences using spanning sets. Further, the authors also introduced pressure of continuous functions in relation to the dynamical systems evolving through holomorphic correspondences. Moreover, they also developed a variational principle using measures supported in space of orbits and explained certain convergence results among the spectral properties of an appropriately defined Ruelle operator. Extending these ideas, we aim to identify a suitable class of measures supported on the Riemann sphere, that in turn will help us in developing an alternative approach towards defining the pressure function. 

The paper is organised thus: In Section \ref{sec:not}, we introduce the main character of the movie, namely holomorphic correspondences defined on the Riemann sphere, fix a few notations and explain the dynamical evolution, that we are interested in. In Section \ref{sec:pfm}, we introduce the push-forwards of measures and use the same effectively in Section \ref{sec:intent} to define the intermediate entropy of a holomorphic correspondence with respect to an ordered pair of measures, carefully chosen and arrive at a result that showcases the relationship between the intermediate entropy of a holomorphic correspondence with respect to an ordered pair of measures and the entropy of an appropriately defined shift map with respect to some invariant measure. In Section \ref{sec:mte}, we define the measure-theoretic entropy of a holomorphic correspondence and state and prove the variational principle, in the current context. Finally, in Section \ref{sec:ruop}, we restrict the dynamics of the holomorphic correspondence on the support of its Dinh-Sibony measure (defined therein), define the Ruelle operator pertaining to this restricted dynamics and prove the existence of a unique measure that satisfies certain convergence properties. 

\section{Some notations} 
\label{sec:not}

We begin this section by defining the set of all permissible paths of iteration of arbitrary length for the holomorphic correspondence $\Gamma$, as written in Definition \ref{holo}. Given $n \in \mathbb{Z}_{+}$, we denote the set of all permissible forward paths of iteration of length $n$ by 
\begin{eqnarray*} 
\mathscr{P}_{n}^{\Gamma} \left( \widehat{\mathbb{C}} \right) & = & \Bigg\{ \left( x_{0}, x^{(1)}_{j_{1}}, \cdots, x^{(n)}_{j_{n}};\; \alpha_{1}, \cdots, \alpha_{n} \right)\ :\ \left( x^{(r - 1)}_{j_{r - 1}}, x^{(r)}_{j_{r}} \right) \in \Gamma_{\alpha_{r}}\ \text{where}\ x^{(0)}_{j_{0}} = x_{0} \in \widehat{\mathbb{C}}, \\ 
& & \hspace{+3cm} \alpha_{r} \in \left\{ 1, 2, \cdots, M \right\} \text{and}\ 1 \le j_{r} \le \lambda_{\alpha_{r}} \left( x^{(r - 1)}_{j_{r - 1}} \right)\ \text{for}\ 1 \le r \le n \Bigg\}, 
\end{eqnarray*} 
where $\lambda_{\alpha_{r}} \left( x^{(r - 1)}_{j_{r - 1}} \right) = \# \left\{ \left( x^{(r - 1)}_{j_{r - 1}}, x^{(r)}_{j_{r}} \right) \in \Gamma_{\alpha_{r}} \right\}$. We denote an arbitrary point in $\mathscr{P}_{n}^{\Gamma} \left( \widehat{\mathbb{C}} \right)$ by $\mathfrak{X}_{n}^{+} \left( x_{0}; \boldsymbol{\alpha} \right)_{\boldsymbol{j}}$ where $x_{0} \in \widehat{\mathbb{C}},\ \boldsymbol{\alpha} = \left( \alpha_{1}, \cdots, \alpha_{n} \right)$ and $\boldsymbol{j} = \left( j_{1}, \cdots, j_{n} \right)$. We define projection maps $\Pi_{\left(r,\, n\right)}^{+} : \mathscr{P}_{n}^{\Gamma} \left( \widehat{\mathbb{C}} \right) \longrightarrow \widehat{\mathbb{C}}$ for any $0 \le r \le n$, and ${\rm Proj}_{\left(r,\, n\right)}^{+} : \mathscr{P}_{n}^{\Gamma} \left( \widehat{\mathbb{C}} \right) \longrightarrow \left\{ 1, 2, \cdots, M \right\}$ for $0 < r \le n$ given by 
\[ \Pi_{\left(r,\, n\right)}^{+} \left( \mathfrak{X}_{n}^{+} \left( x_{0}; \boldsymbol{\alpha} \right)_{\boldsymbol{j}} \right)\ =\ x^{(r)}_{j_{r}}\ \ \text{and}\ \ {\rm Proj}_{\left(r,\, n\right)}^{+} \left( \mathfrak{X}_{n}^{+} \left( x_{0}; \boldsymbol{\alpha} \right)_{\boldsymbol{j}} \right)\ =\ \alpha_{r}. \] 
Endowing the product topology on the product space $\widehat{\mathbb{C}}^{n + 1} \times \left\{ 1, 2, \cdots, M \right\}^{n}$, one may note that the subspace $\mathscr{P}_{n}^{\Gamma} \left( \widehat{\mathbb{C}} \right)$ is compact, since $\widehat{\mathbb{C}}$ is compact and the topology compatible with the metric defined on $\left\{ 1, 2, \cdots, M \right\}^{n}$ should be inherited from the discrete Kronecker delta distance, that one defines on $\left\{ 1, 2, \cdots, M \right\}$. We now define the concept of separated and spanning subsets in $\mathscr{P}_{n}^{\Gamma} \left( \widehat{\mathbb{C}} \right)$, as one may obtain from \cite{ds:2008, ss:2025}. 

\begin{definition} 
Let $\Gamma$ be a holomorphic correspondence defined on $\widehat{\mathbb{C}}$, as written in Definition \ref{holo} and $\mathscr{P}_{n}^{\Gamma} \left( \widehat{\mathbb{C}} \right)$ be the collection of all permissible forward paths of iteration of length $n$ of $\Gamma$. 
\begin{enumerate} 
\item Given $\epsilon > 0$, we say that any two points $\mathfrak{X}_{n}^{+} \left( x_{0}; \boldsymbol{\alpha} \right)_{\boldsymbol{j}}$ and $\mathfrak{X}_{n}^{+} \left( y_{0}; \boldsymbol{\beta} \right)_{\boldsymbol{k}}$ in $\mathscr{P}_{n}^{\Gamma} \left( \widehat{\mathbb{C}} \right)$ are $(n, \epsilon)$-separated if 
\begin{eqnarray*} 
\text{either} & & d_{\widehat{\mathbb{C}}} \left( \Pi_{\left(r,\, n\right)}^{+} \left( \mathfrak{X}_{n}^{+} \left( x_{0}; \boldsymbol{\alpha} \right)_{\boldsymbol{j}} \right),\ \Pi_{\left(r,\, n\right)}^{+} \left( \mathfrak{X}_{n}^{+} \left( y_{0}; \boldsymbol{\beta} \right)_{\boldsymbol{k}} \right) \right)\ \ >\ \ \epsilon\ \ \ \text{for some}\ 0 \le r \le n \\ 
\text{or} & & {\rm Proj}_{\left(r,\, n\right)}^{+} \left( \mathfrak{X}_{n}^{+} \left( x_{0}; \boldsymbol{\alpha} \right)_{\boldsymbol{j}} \right)\ \ne\ {\rm Proj}_{\left(r,\, n\right)}^{+} \left( \mathfrak{X}_{n}^{+} \left( y_{0}; \boldsymbol{\beta} \right)_{\boldsymbol{k}} \right)\ \hspace{+1cm} \text{for some}\ 0 < r \le n. 
\end{eqnarray*} 
A subset $ \mathcal{F} \subseteq \mathscr{P}_{n}^{\Gamma} \left( \widehat{\mathbb{C}} \right)$ is said to be $(n, \epsilon)$-separated if any pair of distinct points in $\mathcal{F}$ are $(n, \epsilon)$-separated. 
\item Given $\epsilon > 0,\ \mathcal{G} \subseteq \mathscr{P}_{n}^{\Gamma} \left( \widehat{\mathbb{C}} \right)$ is said to be a $(n, \epsilon)$-spanning set of $\mathscr{P}_{n}^{\Gamma} \left( \widehat{\mathbb{C}} \right)$ if for every $\mathfrak{X}_{n}^{+} \left( x_{0}; \boldsymbol{\alpha} \right) \in \mathscr{P}_{n}^{\Gamma} \left( \widehat{\mathbb{C}} \right)$ there exists a point $y_{0} \in \widehat{\mathbb{C}}$ such that $\displaystyle{\mathfrak{X}_{n}^{+} \left( y_{0}; \boldsymbol{\alpha} \right) \in \mathcal{G}}$ satisfying 
\[ d_{\widehat{\mathbb{C}}} \left( \Pi_{\left(r,\, n\right)}^{+} \left( \mathfrak{X}_{n}^{+} \left( x_{0}; \boldsymbol{\alpha} \right) \right),\ \Pi_{\left(r,\, n\right)}^{+} \left( \mathfrak{X}_{n}^{+} \left( y_{0}; \boldsymbol{\alpha} \right) \right) \right)\ <\ \epsilon\ \ \ \text{for all}\ 0 \le r \le n. \] 
\end{enumerate} 
\end{definition} 

Making use of the concept of separated sets and spanning sets in $\mathscr{P}_{n}^{\Gamma} \left( \widehat{\mathbb{C}} \right)$, we now define the topological pressure of a real-valued continuous function, say $f \in \mathcal{C} \left( \widehat{\mathbb{C}}, \mathbb{R} \right)$, as written in \cite{ss:2025}. 

\begin{definition}\cite{ss:2025} 
\label{def:cgdpresent}
Let $\Gamma$ be a holomorphic correspondence defined on $\widehat{\mathbb{C}}$, as written in Definition \ref{holo} and $\mathscr{P}_{n}^{\Gamma} \left( \widehat{\mathbb{C}} \right)$ be the collection of all permissible $n$-orbits of $\Gamma$. Suppose $f \in \mathcal{C} \left( \widehat{\mathbb{C}}, \mathbb{R} \right)$. Then, the topological pressure of the function $f$ with respect to the holomorphic correspondence $\Gamma$, denoted by ${\rm Pr} (\Gamma, f)$ is given by 
\begin{eqnarray*} 
& & {\rm Pr} (\Gamma, f) \\ 
& = & \lim_{\epsilon\, \to\, 0}\; \left( \limsup_{n\, \to\, \infty}\; \frac{1}{n} \log \left( \sup\limits_{\begin{minipage}{3cm} \begin{center} \scriptsize{$\mathcal{F}\, \subseteq\, \mathscr{P}_{n}^{\Gamma} \left( \widehat{\mathbb{C}} \right)$} \\ \scriptsize{is an $(n, \epsilon)$} \\ \scriptsize{separated family} \end{center} \end{minipage}} \left\{ \sum_{\mathfrak{X}_{n}^{+} \left( x_{0}; \boldsymbol{\alpha} \right)_{\boldsymbol{j}}\, \in\, \mathcal{F}}\; \left( \prod_{p\, =\, 0}^{n - 1} e^{f \left( \Pi_{\left(p,\, n\right)}^{+} \left( \mathfrak{X}_{n}^{+} \left( x_{0}; \boldsymbol{\alpha} \right)_{\boldsymbol{j}} \right) \right)} \right) \right\} \right) \right) \\ 
& = & \lim_{\epsilon\, \to\, 0}\; \left( \limsup_{n\, \to\, \infty}\; \frac{1}{n} \log \left( \inf\limits_{\begin{minipage}{3cm} \begin{center} \scriptsize{$\mathcal{G}\, \subseteq\, \mathscr{P}_{n}^{\Gamma} \left( \widehat{\mathbb{C}} \right)$} \\ \scriptsize{is an $(n, \epsilon)$} \\ \scriptsize{spanning set} \end{center} \end{minipage}} \left\{ \sum_{\mathfrak{X}_{n}^{+} \left( x_{0}; \boldsymbol{\alpha} \right)_{\boldsymbol{j}}\, \in\, \mathcal{G}}\; \left( \prod_{p\, =\, 0}^{n - 1} e^{f \left( \Pi_{\left(p,\, n\right)}^{+} \left( \mathfrak{X}_{n}^{+} \left( x_{0}; \boldsymbol{\alpha} \right)_{\boldsymbol{j}} \right) \right)} \right) \right\} \right) \right)
\end{eqnarray*} 

and the topological entropy of the holomorphic correspondence is defined as 
\[ h_{{\rm top}} \left( \Gamma \right)\ \ =\ \ {\rm Pr} (\Gamma, 0). \]
\end{definition} 

As one may observe from the definition of topological pressure, we are interested in letting $n \to \infty$, in order to deal with the set of all permissible paths of iteration of infinite length for the correspondence $\Gamma$. We denote this set by $\mathscr{P}^{\Gamma} \left( \widehat{\mathbb{C}} \right)$ as written under. 
\begin{eqnarray*} 
\mathscr{P}^{\Gamma} \left( \widehat{\mathbb{C}} \right) & = & \Bigg\{ \left( x_{0}, x^{(1)}_{j_{1}}, x^{(2)}_{j_{2}}, \cdots;\; \alpha_{1}, \alpha_{2}, \cdots \right)\ :\ \left( x^{(r - 1)}_{j_{r - 1}}, x^{(r)}_{j_{r}} \right) \in \Gamma_{\alpha_{r}}\ \text{where}\ x^{(0)}_{j_{0}} = x_{0} \in \widehat{\mathbb{C}}, \\ 
& & \hspace{+3.5cm} \alpha_{r} \in \left\{ 1, 2, \cdots, M \right\} \text{and}\ 1 \le j_{r} \le \lambda_{\alpha_{r}} \left( x^{(r - 1)}_{j_{r - 1}} \right)\ \text{for}\ r \in \mathbb{Z}_{+} \Bigg\}. 
\end{eqnarray*} 
We denote a point in $\mathscr{P}^{\Gamma} \left( \widehat{\mathbb{C}} \right)$ by $\mathfrak{X}^{+} \left( x_{0}; \boldsymbol{\alpha} \right)_{\boldsymbol{j}}$ where $x_{0} \in \widehat{\mathbb{C}},\ \boldsymbol{\alpha} = \left( \alpha_{1}, \alpha_{2}, \cdots \right)$ and $\boldsymbol{j} = \left( j_{1}, j_{2}, \cdots \right)$. Further, we define a metric $\Delta$ on $\mathscr{P}^{\Gamma} \left( \widehat{\mathbb{C}} \right)$ given by 
\begin{eqnarray} 
\label{eqn:infmet} 
& & \Delta \left( \mathfrak{X}^{+} \left( x_{0}; \boldsymbol{\alpha} \right)_{\boldsymbol{j}}, \mathfrak{X}^{+} \left( y_{0}; \boldsymbol{\beta} \right)_{\boldsymbol{k}} \right) \nonumber \\ 
& = & \max \left\{ \sup_{r\, \in\, \mathbb{Z}_{+} \cup \{ 0 \}} \left\{ \frac{1}{2^{r}} d_{\widehat{\mathbb{C}}} \left( x^{(r)}_{j_{r}}, y^{(r)}_{k_{r}} \right) \right\},\; \sup_{r\, \in\, \mathbb{Z}_{+}} \left\{ \frac{1}{2^{r}} \left( 1 - \delta_{\left( \alpha_{r}, \beta_{r} \right)} \right) \right\} \right\}, 
\end{eqnarray} 
where $d_{\widehat{\mathbb{C}}}$ denotes the spherical metric on the Riemann sphere and $\delta_{(p, q)}$ denotes the discrete Kronecker delta function. One may then observe that $\mathscr{P}^{\Gamma} \left( \widehat{\mathbb{C}} \right)$ is a compact metric space. Analogous to the projection maps $\Pi_{\left(r,\, n\right)}^{+}$ and ${\rm Proj}_{\left(r,\, n\right)}^{+}$ that we defined on $\mathscr{P}_{n}^{\Gamma} \left( \widehat{\mathbb{C}} \right)$, we define projection maps $\Pi_{r} : \mathscr{P}^{\Gamma} \left( \widehat{\mathbb{C}} \right) \longrightarrow \widehat{\mathbb{C}}$ for any $r\, \in\, \mathbb{Z}_{+} \cup \{ 0 \}$, and ${\rm Proj}_{r} : \mathscr{P}^{\Gamma} \left( \widehat{\mathbb{C}} \right) \longrightarrow \left\{ 1, 2, \cdots, M \right\}$ for any $r\, \in\, \mathbb{Z}_{+}$ given by 
\[ \Pi_{r} \left( \mathfrak{X}^{+} \left( x_{0}; \boldsymbol{\alpha} \right)_{\boldsymbol{j}} \right)\ =\ x^{(r)}_{j_{r}}\ \ \text{and}\ \ {\rm Proj}_{r} \left( \mathfrak{X}^{+} \left( x_{0}; \boldsymbol{\alpha} \right)_{\boldsymbol{j}} \right)\ =\ \alpha_{r}. \] 
We now define a shift map $\sigma^{\Gamma}$ pertaining to the holomorphic correspondence $\Gamma$ on the space of permissible paths of iteration of infinite length, $\mathscr{P}^{\Gamma} \left( \widehat{\mathbb{C}} \right)$ by 
\begin{equation} 
\label{sigma} 
\sigma^{\Gamma} \left( \mathfrak{X}^{+} \left( x_{0}; \boldsymbol{\alpha} \right)_{\boldsymbol{j}} \right)\ \ =\ \ \left( x^{(1)}_{j_{1}}, x^{(2)}_{j_{2}}, \cdots;\; \alpha_{2}, \cdots \right)\ \ =\ \ \mathfrak{X}^{+} \left( x^{(1)}_{j_{1}}; \sigma \boldsymbol{\alpha} \right)_{\sigma \boldsymbol{j}}, 
\end{equation} 
where by prefixing $\sigma$ to the infinite-lettered word $\boldsymbol{\alpha}$ or $\boldsymbol{j}$, we mean the infinite-lettered obtained by dropping the first letter of $\boldsymbol{\alpha}$ or $\boldsymbol{j}$, as appropriate. In the topology generated by the metric $\Delta$ on the space $\mathscr{P}^{\Gamma} \left( \widehat{\mathbb{C}} \right)$, one may note that the map $\sigma^{\Gamma}$ is continuous. 

Analogous to the set of permissible paths of forward iteration of length $n \in \mathbb{Z}_{+}$, we also define permissible paths of backward iteration of length $n \in \mathbb{Z}_{+}$. Towards that end, consider 
\begin{eqnarray*} 
\mathscr{Q}_{n}^{\Gamma} \left( \widehat{\mathbb{C}} \right) & = & \Bigg\{ \left( y^{(-n)}_{k_{n}}, \cdots, y^{(-1)}_{k_{1}}, y_{0};\; \beta_{n - 1}, \cdots, \beta_{0} \right)\ :\ \left( y^{(-r)}_{k_{r}}, y^{(- (r - 1))}_{k_{r - 1}} \right) \in \Gamma_{\beta_{r - 1}}\ \text{where} \\ 
& & \hspace{+1.5cm} y^{(0)}_{k_{0}} = y_{0} \in \widehat{\mathbb{C}},\ \beta_{r - 1} \in \left\{ 1, 2, \cdots, M \right\}\ \text{and}\ 1 \le k_{r} \le \delta_{\beta_{r - 1}} \left( y^{(- (r - 1))}_{k_{r - 1}} \right) \\ 
& & \hspace{+9cm} \text{for}\ 0 \le r - 1 \le n - 1 \Bigg\}, 
\end{eqnarray*} 
where $\delta_{\beta_{r - 1}} \left( y^{(- (r - 1))}_{k_{r - 1}} \right) = \# \left\{ \left( y^{(-r)}_{k_{r}}, y^{(- (r - 1))}_{k_{r - 1}} \right) \in \Gamma_{\beta_{r - 1}} \right\}$. We denote any point in $\mathscr{Q}_{n}^{\Gamma} \left( \widehat{\mathbb{C}} \right)$ by $\mathfrak{X}_{n}^{-} \left( y_{0}; \boldsymbol{\beta} \right)_{\boldsymbol{k}}$ where $y_{0} \in \widehat{\mathbb{C}},\ \boldsymbol{\beta} = \left( \beta_{n - 1}, \cdots, \beta_{0} \right)$ and $\boldsymbol{k} = \left( k_{n}, \cdots, k_{1} \right)$. Moreover, we define projection maps $\Pi_{\left(r,\, n\right)}^{-} : \mathscr{Q}_{n}^{\Gamma} \left( \widehat{\mathbb{C}} \right) \longrightarrow \widehat{\mathbb{C}}$ for any $0 \le r \le n$, and ${\rm Proj}_{\left(r,\, n\right)}^{-} : \mathscr{Q}_{n}^{\Gamma} \left( \widehat{\mathbb{C}} \right) \longrightarrow \left\{ 1, 2, \cdots, M \right\}$ for $0 \le r < n$ given by 
\[ \Pi_{\left(r,\, n\right)}^{-} \left( \mathfrak{X}_{n}^{-} \left( y_{0}; \boldsymbol{\beta} \right)_{\boldsymbol{k}} \right)\ =\ y^{(-r)}_{k_{r}}\ \ \text{and}\ \ {\rm Proj}_{\left(r,\, n\right)}^{-} \left( \mathfrak{X}_{n}^{-} \left( y_{0}; \boldsymbol{\beta} \right)_{\boldsymbol{k}} \right)\ =\ \beta_{r}. \] 

For more details on the set of permissible paths of iteration, the metric defined therein, the shift map and the projection maps, as defined in this section, readers are referred to \cite{bs:2021, ss:2025}. 

\section{The push-forward of a measure} 
\label{sec:pfm} 

Suppose $\mathscr{M} \left( \mathscr{P}^{\Gamma} \left( \widehat{\mathbb{C}} \right) \right)$ denotes the collection of all probability measures supported on $\mathscr{P}^{\Gamma} \left( \widehat{\mathbb{C}} \right)$. Since $\mathscr{P}^{\Gamma} \left( \widehat{\mathbb{C}} \right)$ is a compact metric space, we can accord a metric (compatible with the weak*-topology) on $\mathscr{M} \left( \mathscr{P}^{\Gamma} \left( \widehat{\mathbb{C}} \right) \right)$, given by 
\[ d \left( \mu_{1}, \mu_{2} \right) = \sum_{n\, \ge\, 1} \frac{1}{2^{n}} \left| \int f_{n} \mathrm{d}\mu_{1} - \int f_{n} \mathrm{d}\mu_{2} \right|, \] 
where $\left\{ f_{n} \right\}_{n\, \ge\, 1}$ is a dense sequence of continuous real-valued functions defined on $\mathscr{P}^{\Gamma} \left( \widehat{\mathbb{C}} \right)$. 

Let $\mathscr{M}_{\sigma^{\Gamma}} \left( \mathscr{P}^{\Gamma} \left( \widehat{\mathbb{C}} \right) \right) \subseteq \mathscr{M} \left( \mathscr{P}^{\Gamma} \left( \widehat{\mathbb{C}} \right) \right)$ denote the set of all $\sigma^{\Gamma}$-invariant probability measures. Since $\sigma^{\Gamma}$ is a continuous map defined on the compact metric space $\mathscr{P}^{\Gamma} \left( \widehat{\mathbb{C}} \right)$, an elementary result from \cite{pw:1975} says that $\mathscr{M}_{\sigma^{\Gamma}} \left( \mathscr{P}^{\Gamma} \left( \widehat{\mathbb{C}} \right) \right)$ is non-empty, convex and compact. 

\begin{definition} 
For any measure $\mu \in \mathscr{M}_{\sigma^{\Gamma}} \left( \mathscr{P}^{\Gamma} \left( \widehat{\mathbb{C}} \right) \right)$, we define its \emph{push-forward measure} supported on $\widehat{\mathbb{C}}$, denoted by $\left( \Pi_{r} \right)_{*} \mu$, under the projection $\Pi_{r}$ for any $r \in \mathbb{Z}_{+} \cup \{ 0 \}$ as, 
\[ \left( \left( \Pi_{r} \right)_{*} \mu \right) (B)\ \ =\ \ \mu \left( \Pi_{r}^{-1} B \right),\ \ \ \text{for any Borel set}\ B \subset \widehat{\mathbb{C}}. \] 
\end{definition} 

Observe that this definition entails that for any $r \in \mathbb{Z}_{+} \cup \{ 0 \}$ and for an arbitrary choice of $f \in \mathcal{C} \left( \widehat{\mathbb{C}}, \mathbb{R} \right)$, 
\begin{eqnarray*} 
\int_{\widehat{\mathbb{C}}} f \mathrm{d}\left( \Pi_{r} \right)_{*} \mu & = & \int_{\mathscr{P}^{\Gamma} \left( \widehat{\mathbb{C}} \right)} \left( f \circ \Pi_{r} \right) \mathrm{d}\mu\ =\ \int_{\mathscr{P}^{\Gamma} \left( \widehat{\mathbb{C}} \right)} \left( f \circ \Pi_{r} \circ \sigma^{\Gamma} \right) \mathrm{d}\mu \\ 
& = & \int_{\mathscr{P}^{\Gamma} \left( \widehat{\mathbb{C}} \right)} \left( f \circ \Pi_{r + 1} \right) \mathrm{d}\mu\ = \int_{\widehat{\mathbb{C}}} f \mathrm{d}\left( \Pi_{r + 1} \right)_{*} \mu. 
\end{eqnarray*} 
Thus, $\left( \Pi_{r} \right)_{*} \mu = \left( \Pi_{r + 1} \right)_{*} \mu$, almost everywhere on $\widehat{\mathbb{C}}$ and hence, it is sufficient to work with the measure $\left( \Pi_{0} \right)_{*} \mu$. However, note that there is a possibility that $\left( \Pi_{0} \right)_{*} \mu_{1} \equiv \left( \Pi_{0} \right)_{*} \mu_{2}$ on $\widehat{\mathbb{C}}$, even when $\mu_{1} \ne \mu_{2}$ on $\mathscr{P}^{\Gamma} \left( \widehat{\mathbb{C}} \right)$. Let $\mathscr{S}^{\Gamma} \subseteq \mathscr{M} \left( \widehat{\mathbb{C}} \right)$ be the set of all measures such that 
\[ \mathscr{S}^{\Gamma}\ \ =\ \ \left\{ \nu \in \mathscr{M} \left( \widehat{\mathbb{C}} \right) : \nu = \left( \Pi_{0} \right)_{*} \mu\ \text{for some}\ \mu \in \mathscr{M}_{\sigma^{\Gamma}} \left( \mathscr{P}^{\Gamma} \left( \widehat{\mathbb{C}} \right) \right) \right\}. \]

\begin{lemma} 
$\mathscr{S}^{\Gamma}$ is a weak*-compact, convex subset of $\mathscr{M} \left( \widehat{\mathbb{C}} \right)$. 
\end{lemma} 

\begin{proof} 
In order to prove $\mathscr{S}^{\Gamma}$ is weak*-compact, it is sufficient to consider a sequence of measures, say $\left\{ \nu_{n} \right\}_{n\, \in\, \mathbb{Z}_{+}}$ in $\mathscr{S}^{\Gamma}$ that converges to $\nu_{0} \in \mathscr{M} \left( \widehat{\mathbb{C}} \right)$ and establish that $\nu_{0} \in \mathscr{S}^{\Gamma}$. Corresponding to every element $\nu_{n}$, we can choose some measure $\mu_{n} \in \mathscr{M}_{\sigma^{\Gamma}} \left( \mathscr{P}^{\Gamma} \left( \widehat{\mathbb{C}} \right) \right)$ such that $\nu_{n} = \left( \Pi_{0} \right)_{*} \mu_{n}$. Since $\mathscr{M}_{\sigma^{\Gamma}} \left( \mathscr{P}^{\Gamma} \left( \widehat{\mathbb{C}} \right) \right)$ is a compact space, we know that the sequence of measures $\left\{ \mu_{n} \right\}_{n\, \in\, \mathbb{Z}_{+}}$ has a subsequence, say $\left\{ \mu_{n_{l}} \right\}_{l\, \in\, \mathbb{Z}_{+}}$ that converges to some measure, say $\mu_{0} \in \mathscr{M}_{\sigma^{\Gamma}} \left( \mathscr{P}^{\Gamma} \left( \widehat{\mathbb{C}} \right) \right)$. Hence, the ideal candidate for the measure $\nu_{0} \in \mathscr{M} \left( \widehat{\mathbb{C}} \right)$ is the measure $\left( \Pi_{0} \right)_{*} \mu_{0} \in \mathscr{S}^{\Gamma}$. We now prove that $\nu_{0} \equiv \left( \Pi_{0} \right)_{*} \mu_{0}$. 

Let $f \in \mathcal{C} \left( \widehat{\mathbb{C}}, \mathbb{R} \right)$ be some arbitrary function. Since 
\begin{eqnarray*} 
\left| \int_{\widehat{\mathbb{C}}} f \mathrm{d}\left( \Pi_{0} \right)_{*} \mu_{0} - \int_{\widehat{\mathbb{C}}} f \mathrm{d}\nu_{0} \right| & \le & \left| \int_{\mathscr{P}^{\Gamma} \left( \widehat{\mathbb{C}} \right)} \left( f \circ \Pi_{0} \right) \mathrm{d}\mu_{0} - \int_{\mathscr{P}^{\Gamma} \left( \widehat{\mathbb{C}} \right)} \left( f \circ \Pi_{0} \right) \mathrm{d}\mu_{n_{l}} \right| \\ 
& & \hspace{+3cm} + \left| \int_{\widehat{\mathbb{C}}} f \mathrm{d}\left( \Pi_{0} \right)_{*} \mu_{n_{l}} - \int_{\widehat{\mathbb{C}}} f \mathrm{d}\nu_{0} \right|, 
\end{eqnarray*} 
one can see that each of the terms in the right hand side of the above inequality can be made as small as required, implying $\left( \Pi_{0} \right)_{*} \mu_{0} = \nu_{0}$. Further, the convexity of the space $\mathscr{M}_{\sigma^{\Gamma}} \left( \mathscr{P}^{\Gamma} \left( \widehat{\mathbb{C}} \right) \right)$ assures us that $\mathscr{S}^{\Gamma}$ is a convex subset of $\mathscr{M} \left( \widehat{\mathbb{C}} \right)$. 
\end{proof} 

\section{An intermediate entropy} 
\label{sec:intent} 

Let $\nu \in \mathscr{S}^{\Gamma}$. Suppose $\mathcal{Q}$ is a finite $\nu$-measurable partition of $\widehat{\mathbb{C}}$, say $\mathcal{Q} = \left\{ Q_{1}, Q_{2}, \cdots, Q_{s} \right\}$. Then, corresponding to the partition $\mathcal{Q}$, the collection of sets denoted by $\mathcal{Q}^{\Gamma}$ and defined as 
\[ \mathcal{Q}^{\Gamma}\ \ =\ \ \left\{ \Pi_{0}^{-1} \left( Q_{i} \right) \cap {\rm Proj}_{1}^{-1} \left( \left\{ t \right\} \right) : 1 \le i \le s, 1 \le t \le M \right\} \] 
forms a $\mu_{i}$-measurable partition of $\mathscr{P}^{\Gamma} \left( \widehat{\mathbb{C}} \right)$ for all measures $\mu_{i} \in \mathscr{M}_{\sigma^{\Gamma}} \left( \mathscr{P}^{\Gamma} \left( \widehat{\mathbb{C}} \right) \right)$ satisfying $\nu = \left( \Pi_{0} \right)_{*} \mu_{i}$. We now define the information function of the partition $\mathcal{Q}$ with respect to the ordered pair of measures $\left( \nu,\, \mu \right)$ for $\nu \in \mathscr{S}^{\Gamma}$ and $\mu \in \mathscr{M}_{\sigma^{\Gamma}} \left( \mathscr{P}^{\Gamma} \left( \widehat{\mathbb{C}} \right) \right)$ satisfying $\nu = \left( \Pi_{0} \right)_{*} \mu$ as $H_{\left( \nu,\, \mu \right)} \left( \mathcal{Q} \right) = H_{\mu} \left( \mathcal{Q}^{\Gamma} \right)$ where for any measurable collection of subsets, say $\mathcal{A} = \left\{ A_{1}, A_{2}, \cdots, A_{s} \right\}$ of $\mathscr{P}^{\Gamma} \left( \widehat{\mathbb{C}} \right)$ we define 
\[ H_{\mu} \left( \left\{ A_{1}, A_{2}, \cdots, A_{s} \right\} \right)\ \ =\ \ - \sum_{i\, =\, 1}^{s} \mu (A_{i}) \log \mu (A_{i}). \] 
Further, if $\mathcal{B} = \left\{ B_{1}, B_{2}, \cdots, B_{s'} \right\}$ is another collection of subsets of $\mathscr{P}^{\Gamma} \left( \widehat{\mathbb{C}} \right)$, we define 
\[ \mathcal{A} \vee \mathcal{B}\ \ =\ \ \left\{ A_{i} \cap B_{i'}\ :\ A_{i} \in \mathcal{A}\ \text{and}\ B_{i'} \in \mathcal{B}\ \ \text{for}\ \ 1 \le i \le s,\ 1 \le i' \le s' \right\}. \]

Then, the information function of the partition $\mathcal{Q}$ with respect to the ordered pair of measures $\left( \nu,\, \mu \right)$ pertaining to the holomorphic correspondence $\Gamma$ is now defined as 
\[ H_{\left( \nu,\, \mu \right)} \left( \mathcal{Q}; \Gamma \right)\ \ =\ \ \lim_{n\, \to\, \infty} \frac{1}{n} H_{\mu} \left( \bigvee_{p\, =\, 0}^{n - 1} \left\{ \Pi_{p}^{-1} \left( Q_{i} \right) \cap {\rm Proj}_{p + 1}^{-1} \left( \left\{ t \right\} \right) : 1 \le i \le s,\ 1 \le t \le M \right\} \right). \] 
Finally, we take the supremum over all possible $\nu$-measurable partitions $\mathcal{Q}$ of $\widehat{\mathbb{C}}$ and define the \emph{intermediate entropy} of the correspondence $\Gamma$ with respect to the ordered pair of measures $\left( \nu,\, \mu \right)$, given by 
\[ h_{\left( \nu,\, \mu \right)} \left( \Gamma \right)\ \ =\ \ \sup_{\mathcal{Q}} H_{\left( \nu,\, \mu \right)} \left( \mathcal{Q}; \Gamma \right). \] 

With these definitions, we now state and prove a theorem that relates the intermediate entropy of the correspondence $\Gamma$ with respect to the ordered pair of measures $\left( \nu,\, \mu \right)$ and the measure-theoretic entropy of the shift map $\sigma^{\Gamma}$ with respect to the measure $\mu$. 

\begin{theorem}
\label{Entropy} 
Let $\Gamma$ be a holomorphic correspondence defined on $\widehat{\mathbb{C}}$, as written in Definition \ref{holo}, whose permissible paths of iteration is given by $\mathscr{P}^{\Gamma} \left( \widehat{\mathbb{C}} \right)$. Let $\sigma^{\Gamma}$ be the left shift map defined on $\mathscr{P}^{\Gamma} \left( \widehat{\mathbb{C}} \right)$. Consider the ordered pair of measures $\left( \nu,\, \mu \right)$ for $\nu \in \mathscr{S}^{\Gamma}$ and $\mu \in \mathscr{M}_{\sigma^{\Gamma}} \left( \mathscr{P}^{\Gamma} \left( \widehat{\mathbb{C}} \right) \right)$ satisfying $\nu = \left( \Pi_{0} \right)_{*} \mu$. Then, 
\[ h_{\left( \nu,\, \mu \right)} \left( \Gamma \right)\ \ =\ \ h_{\mu} \left( \sigma^{\Gamma} \right). \] 
\end{theorem} 

\begin{proof} 
Since, $\sigma^{\Gamma}$ is a continuous map defined on the compact metric space $\mathscr{P}^{\Gamma} \left( \widehat{\mathbb{C}} \right)$, the measure-theoretic entropy of $\sigma^{\Gamma}$ with respect to $\mu$ is given by 
\[ \sup_{\mathcal{P}} \left\{ \lim_{n\, \to\, \infty} \frac{1}{n} H_{\mu} \left( \bigvee_{p\, =\, 0}^{n - 1} \left\{ \left( \sigma^{\Gamma} \right)^{-p} \left( \mathcal{P} \right)   \right\} \right) \right\}, \]
where the supremum is taken over all finite $\mu$-measurable  partitions $\mathcal{P}$ of $\mathscr{P}^{\Gamma} \left( \widehat{\mathbb{C}} \right)$. Further, for any partition $\mathcal{P} = \{ P_{1}, P_{2}, \cdots, P_{s} \}$, we have 
\[ \left( \sigma^{\Gamma} \right)^{-p} \left( \mathcal{P} \right)\ \ =\ \ \left\{ \left( \sigma^{\Gamma} \right)^{-p}(P_{1}), \left( \sigma^{\Gamma} \right)^{-p}(P_{2}), \cdots, \left( \sigma^{\Gamma} \right)^{-p}(P_{s}) \right\}. \]

As written earlier, we know that if $\mathcal{Q} = \left\{ Q_{1}, Q_{2}, \cdots, Q_{s} \right\}$ is a finite $\nu$-measurable partition of $\widehat{\mathbb{C}}$, then there exists a finite $\mu$-measurable partition of $\mathscr{P}^{\Gamma} \left( \widehat{\mathbb{C}} \right)$, corresponding to $\mathcal{Q}$, denoted by $\mathcal{Q}^{\Gamma}$. Then it is easy to verify that 
\begin{eqnarray*} 
& & H_{\mu} \left( \bigvee_{p\, =\, 0}^{n - 1} \left\{ \Pi_{p}^{-1} \left( Q_{i} \right) \cap {\rm Proj}_{p + 1}^{-1} \left( \left\{ t \right\} \right)\ :\ 1 \le i \le s,\ 1 \le t \le M \right\} \right) \\ 
& = & H_{\mu} \left( \bigvee_{p\, =\, 0}^{n - 1} \left\{ \left( \sigma^{\Gamma} \right)^{-p} \left( \mathcal{Q}^\Gamma \right) \right\} \right). 
\end{eqnarray*} 

We now state a lemma from \cite{pw:1975}, without proof however, in the language that we have so far built in this paper.

\begin{lemma}
\label{diam0}
Let  $\sigma^{\Gamma} : \mathscr{P}^{\Gamma} \left( \widehat{\mathbb{C}} \right) \longrightarrow \mathscr{P}^{\Gamma} \left( \widehat{\mathbb{C}} \right)$ be the continuous left shift map as defined in Equation \eqref{sigma}. Let $\mu \in \mathscr{M}_{\sigma^{\Gamma}} \left( \mathscr{P}^{\Gamma} \left( \widehat{\mathbb{C}} \right) \right)$. 
\begin{enumerate} 
\item[(a)] Suppose $\mathcal{P}$ is a $\mu$-measurable partition of $\mathscr{P}^{\Gamma} \left( \widehat{\mathbb{C}} \right)$. Then, 
\[ \lim_{n\, \to\, \infty} \frac{1}{n} H_{\mu} \left( \bigvee_{p\, =\, 0}^{n - 1} \left\{ \left( \sigma^{\Gamma} \right)^{-p} \left( \mathcal{P} \right) \right\} \right)\ \ =\ \ \lim_{n\, \to\, \infty} \frac{1}{n} H_{\mu} \left( \bigvee_{p\, =\, 0}^{n - 1} \left\{ \left( \sigma^{\Gamma} \right)^{-p} \left( \bigvee_{i\, =\, 0}^{p' - 1} \left\{ \left( \sigma^{\Gamma} \right)^{-i} \left( \mathcal{P} \right) \right\} \right) \right\} \right). \] 
\item[(b)] Suppose $\left\{ \mathcal{P}_{s'} \right\}_{s'\, \in\, \mathbb{Z}_{+}}$ is a sequence of $\mu$-measurable partitions of $\mathscr{P}^{\Gamma} \left( \widehat{\mathbb{C}} \right)$ satisfying ${\rm diam} \left( \mathcal{P}_{s'} \right) = \max \left\{ {\rm diam} (P) : P \in \mathcal{P}_{s'} \right\} \to 0$, as $s' \to \infty$. Then, the measure-theoretic entropy of the shift map $\sigma^{\Gamma}$ pertaining to the holomorphic correspondence $\Gamma$ with respect to the measure $\mu$ is given by 
\[ h_{\mu} \left( \sigma^{\Gamma} \right)\ \ =\ \ \lim_{s'\, \to\, \infty} \lim_{n\, \to\, \infty} \frac{1}{n} H_{\mu} \left( \bigvee_{p\, =\, 0}^{n - 1} \left\{ \left( \sigma^{\Gamma} \right)^{-p} \left( \mathcal{P}_{s'} \right) \right\} \right). \] 
\end{enumerate} 
\end{lemma} 

We now continue with the proof of Theorem \ref{Entropy}. Let $\left\{ \mathcal{P}_{s'} \right\}_{s'\, \in\, \mathbb{Z}_{+}}$ denote a collection of finite $\nu$-measurable partitions of $\widehat{\mathbb{C}}$ where $\mathcal{P}_{s'} = \left\{ P_{(s', 1)}, P_{(s', 2)}, \cdots, P_{(s', s(s'))} \right\}$ with ${\rm diam} \left( \mathcal{P}_{s'} \right) \le p_{s'}$ and $p_{s'} \to 0$ as $s' \to \infty$. Then, corresponding to $\mathcal{P}_{s'}$, there exists a positive integer $q_{s'}$ such that 
\[ \mathcal{P}_{(s',\, q_{s'})}^{\Gamma}\ \ =\ \ \bigvee_{r\, =\, 0}^{q_{s'}} \left\{ \left( \sigma^{\Gamma} \right)^{-r} \left( \Pi_{0}^{-1} \left( P_{(s', i)} \right) \cap {\rm Proj}_{1}^{-1} \left( \left\{ t \right\} \right) \right)\ :\ 1 \le i \le s(s'), 1 \le t \le M \right\} \]
is a $\mu$-measurable partition of $\mathscr{P}^{\Gamma} \left( \widehat{\mathbb{C}} \right)$ corresponding to $\mathcal{P}_{s'}$ that satisfies the condition ${\rm diam} \left( \mathcal{P}_{(s',\, q_{s'})}^{\Gamma} \right) \le p_{s'}$. As $s' \to \infty$, we obtain from Lemma \eqref {diam0} that 
\begin{eqnarray*} 
h_{\mu} \left( \sigma^{\Gamma} \right) & = & \lim_{s'\, \to\, \infty} \lim_{n\, \to\, \infty} \frac{1}{n} H_{\mu} \left( \bigvee_{p\, =\, 0}^{n - 1} \left\{ \left( \sigma^{\Gamma} \right)^{-p} \left( \mathcal{P}_{s'}^{\Gamma} \right) \right\} \right) \\ 
& = & \lim_{s'\, \to\, \infty} \lim_{q_{s'}\, \to\, \infty} \frac{1}{q_{s'}} H_{\mu} \left( \bigvee_{p\, =\, 0}^{q_{s'} - 1} \left\{ \left( \sigma^{\Gamma} \right)^{-p} \left( \mathcal{P}_{(s',\, q_{s'})}^{\Gamma} \right) \right\} \right) \\ 
& = & \lim_{s'\, \to\, \infty} \lim_{n \to \infty} \frac{1}{n} H_{\mu} \left( \bigvee_{p\, =\, 0}^{n - 1} \left\{ \Pi_{p}^{-1} \left( P_{(s', i)} \right) \cap {\rm Proj}_{p + 1}^{-1} \left( \left\{ t \right\} \right) : 1 \le i \le s(s'),\ 1 \le t \le M \right\} \right) \\ 
& = & \lim_{s' \to \infty} H_{\left( \nu,\, \mu \right)} \left( \mathcal{P}_{s'} \right) \\ 
& \le & h_{\left( \nu,\, \mu \right)} \left( \Gamma \right). 
\end{eqnarray*} 
The other side inequality is quite trivial and hence the proof.
 \end{proof} 

\section{Measure-theoretic entropy} 
\label{sec:mte} 

We begin this section with the definition of the measure-theoretic entropy of the holomorphic correspondence $\Gamma$ defined on $\widehat{\mathbb{C}}$, as written in Definition \ref{holo}, with respect to a measure $\nu \in \mathscr{S}^{\Gamma}$. 

\begin{definition} 
Let $\Gamma$ be a holomorphic correspondence defined on $\widehat{\mathbb{C}}$, as written in Definition \ref{holo}, whose permissible paths of iteration is given by $\mathscr{P}^{\Gamma} \left( \widehat{\mathbb{C}} \right)$. Let $\nu \in \mathscr{S}^{\Gamma}$. Then, the \emph{measure-theoretic entropy} of the holomorphic correspondence $\Gamma$ is defined as 
\[ h_{\nu} \left( \Gamma \right)\ \ =\ \ \sup \left\{ h_{\left( \nu,\, \mu_{n} \right)} \left( \Gamma \right)\ :\ \mu_{n} \in \mathscr{M}_{\sigma^{\Gamma}} \left( \mathscr{P}^{\Gamma} \left( \widehat{\mathbb{C}} \right) \right)\ \text{satisfying}\ \nu = \left( \Pi_{0} \right)_{*} \mu_{n} \right\}. \] 
\end{definition} 

Having defined the measure-theoretic entropy of the holomorphic correspondence $\Gamma$, we now write the variational principle as follows. 

\begin{theorem}[Variational principle for a holomorphic correspondence] 
\label{thm:vphc}
Let $\Gamma$ be a holomorphic correspondence defined on $\widehat{\mathbb{C}}$, as written in Definition \ref{holo} and $f \in \mathcal{C} \left( \widehat{\mathbb{C}}, \mathbb{R} \right)$. Then, the pressure of the function $f$ can alternatively be characterized as 
\[ {\rm Pr} (\Gamma, f)\ \ =\ \ \sup_{\nu\, \in\, \mathscr{S}^{\Gamma}} \left\{ h_{\nu} \left( \Gamma \right) + \int f \mathrm{d}\nu \right\}. \]
\end{theorem} 

\begin{proof} 
We know from Theorem 5.3 in \cite{ss:2025} that 
\[ {\rm Pr} (\Gamma, f)\ =\ {\rm Pr} \left( \sigma^{\Gamma}, f \circ \Pi_{0} \right)\ =\ \sup_{\mu\, \in\, \mathscr{M}_{\sigma^{\Gamma}} \left( \mathscr{P}^{\Gamma} \left( \widehat{\mathbb{C}} \right) \right)} \left\{ h_{\mu} \left( \sigma^{\Gamma} \right) + \int f \circ \Pi_{0} \mathrm{d}\mu \right\}. \] 
Hence, 
\begin{eqnarray*} 
{\rm Pr} (\Gamma, f) & = & \sup_{\mu\, \in\, \mathscr{M}_{\sigma^{\Gamma}} \left( \mathscr{P}^{\Gamma} \left( \widehat{\mathbb{C}} \right) \right)} \left\{ h_{\left( \left( \Pi_{0} \right)_{*} \mu,\, \mu \right)} (\Gamma) + \int f \mathrm{d}\left( \Pi_{0} \right)_{*} \mu \right\} \\ 
& = & \sup_{\nu\, \in\, \mathscr{S}^{\Gamma}}\; \sup_{\mu\; :\; \left( \Pi_{0} \right)_{*} \mu\, =\, \nu} \left\{ h_{\left( \nu,\, \mu \right)} (\Gamma) + \int f \mathrm{d}\nu \right\} \\ 
& = & \sup_{\nu\, \in\, \mathscr{S}^{\Gamma}} \left\{ h_{\nu} (\Gamma) + \int f \mathrm{d}\nu \right\}. 
\end{eqnarray*} 
\end{proof} 

\begin{corollary} 
The topological entropy of a holomorphic correspondence $\Gamma$, as written in Definition \ref{holo}, can be alternatively characterised as the supremum of all the measure-theoretic entropies of $\Gamma$, where the supremum is taken over the set $\mathscr{S}^{\Gamma}$, {\it i.e.}, 
\[ h_{{\rm top}} \left( \Gamma \right)\ \ =\ \ \sup \Big\{ h_{\nu} \left( \Gamma \right)\ :\ \nu\, \in\, \mathscr{S}^{\Gamma} \Big\}. \]
\end{corollary} 

\begin{theorem} 
\label{thm:usc} 
Let $\Gamma$ be the holomorphic correspondence on $\widehat{\mathbb{C}}$ as written in Definition \ref{holo} with $h_{{\rm top}} (\Gamma) < \infty$. Suppose $\nu \in \mathscr{S}^{\Gamma}$. Then, $\displaystyle{h_{\nu} (\Gamma) = \inf\limits_{f\, \in\, \mathcal{C} \left( \widehat{\mathbb{C}},\, \mathbb{R} \right)} \left\{ {\rm Pr} (\Gamma, f) - \int f\, \mathrm{d}\nu \right\}}$ {\it iff} the map $\nu \longmapsto h_{\nu} (\Gamma)$ is upper semicontinuous at $\nu$. 
\end{theorem}

One may observe that the statement of Theorem \ref{thm:usc}, written above and its proof written below follow the exact same lines as that of the classical case of dynamical systems determined by a continuous map defined on some compact space. 

\begin{proof} 
Suppose $\displaystyle{h_{\nu} (\Gamma) = \inf\limits_{f\, \in\, \mathcal{C} \left( \widehat{\mathbb{C}},\, \mathbb{R} \right)} \left\{ {\rm Pr} (\Gamma, f) - \int f\, \mathrm{d}\nu \right\}}$. Given $\epsilon > 0$, choose $g \in \mathcal{C} \left( \widehat{\mathbb{C}}, \mathbb{R} \right)$ such that $\displaystyle{{\rm Pr} (\Gamma, g) - \int g\, \mathrm{d}\nu < h_{\nu} (\Gamma) + \dfrac{\epsilon}{2}}$. Consider the neighbourhood around $\nu$ in $\mathscr{S}^{\Gamma}$ pertaining to the function $g$ prescribed by 
\[ V_{g} \left( \nu,\; \dfrac{\epsilon}{2} \right)\ \ =\ \ \left\{ \nu' \in \mathscr{S}^{\Gamma}\ :\ \left| \int g \mathrm{d}\nu - \int g \mathrm{d}\nu' \right|\ <\ \frac{\epsilon}{2} \right\}. \]
Then, by the variational principle for holomorphic correspondences, as written in Theorem \ref{thm:vphc}, we have 
\[ h_{\nu'} (\Gamma)\ \le\ {\rm Pr} (\Gamma, g)\; -\; \int g\, \mathrm{d}\nu'\ <\ {\rm Pr} (\Gamma, g)\; -\; \int g\, \mathrm{d}\nu\; +\; \frac{\epsilon}{2}\ <\ h_{\nu} (\Gamma)\; +\; \epsilon,\ \ \ \forall \nu' \in V_{g} \left( \nu,\; \dfrac{\epsilon}{2} \right). \] 
Thus the entropy map is upper semicontinuous at $\nu$. 

On the other hand, we now assume that the entropy map is upper semicontinuous at some $\nu \in \mathscr{S}^{\Gamma}$. Then, by the variational principle, we have $\displaystyle{h_{\nu} (\Gamma) \le \inf\limits_{f\, \in\, \mathcal{C} \left( \widehat{\mathbb{C}}, \mathbb{R} \right)} \left\{ {\rm Pr} (\Gamma, f) - \int f\, \mathrm{d}\nu \right\}}$. The proof is complete if $\displaystyle{h_{\nu} (\Gamma) \ge \inf\limits_{f\, \in\, \mathcal{C} \left( \widehat{\mathbb{C}}, \mathbb{R} \right)} \left\{ {\rm Pr} (\Gamma, f) - \int f\, \mathrm{d}\nu \right\}}$. Towards that end, consider a convex set $\mathcal{K} = \left\{ (\nu', u) \in \mathscr{S}^{\Gamma} \times \mathbb{R} : 0 \le u \le h_{\nu'} (\Gamma) \right\}$. Then, $(\nu, u') \notin \mathcal{K}$ whenever $u' > h_{\nu} (\Gamma)$. The upper semicontinuity of the entropy map and our definition of $\mathcal{K}$ ensures that $\mathcal{K}$ is closed. Hence, by an application of the Hahn-Banach separation theorem, we obtain the existence of a continuous linear functional, say $F : \mathcal{C} \left( \widehat{\mathbb{C}}, \mathbb{R} \right)^{*} \times \mathbb{R} \to \mathbb{R}$ such that $F \left( \nu', u \right) < F \left( \nu, u' \right)$ for every $\left( \nu', u \right) \in \mathcal{K}$. Since $\mathcal{C} \left( \widehat{\mathbb{C}}, \mathbb{R} \right)^{*}$ is equipped with the weak$^{*}$-topology, $F$ can be written as $\displaystyle{F \left( \nu', u \right) = \int f\, \mathrm{d}\nu' + cu}$, for some $f \in \mathcal{C} \left( \widehat{\mathbb{C}}, \mathbb{R} \right)$ and $c \in \mathbb{R}$. 

Thus, for all $\left( \nu', u \right) \in \mathcal{K}$, we have $\displaystyle{\int f\, \mathrm{d}\nu' + cu < \int f\, \mathrm{d}\nu + cu'}$ and in particular, 
\[ \int f\, \mathrm{d}\nu'\ +\ c h_{\nu'} (\Gamma)\ \ <\ \ \int f\, \mathrm{d}\nu\ +\ cu',\ \ \ \text{whenever}\ \nu' \in \mathscr{S}^{\Gamma}, \] 
which results in $c h_{\nu} (\Gamma) < cu'$, since $\nu \in \mathscr{S}^{\Gamma}$ and thus, $c > 0$. Then, 
\[ h_{\nu'} (\Gamma)\ \ <\ \ u' + \frac{1}{c} \int f\, \mathrm{d}\nu - \frac{1}{c} \int f\, \mathrm{d}\nu'. \] 
Further, by variational principle, we have $\displaystyle{{\rm Pr} \left( \Gamma, \frac{f}{c} \right) \le u' + \frac{1}{c} \int f\, \mathrm{d}\nu}$. Since this inequality holds for all $u' > h_{\nu} (\Gamma)$, we have 
\[ h_{\nu} (\Gamma) \ge \inf_{g\, \in\, \mathcal{C} \left( \widehat{\mathbb{C}}, \mathbb{R} \right)} \left\{ {\rm Pr} (\Gamma, g) - \int g\, \mathrm{d}\nu \right\}, \] 
thereby completing the proof. 
\end{proof}

\section{The Ruelle operator} 
\label{sec:ruop} 

We begin this section by recalling a result from \cite{bs:2016} for holomorphic correspondences $\Gamma$, as written in Definition \ref{holo}, that satisfy $d_{{\rm fwd}} < d_{{\rm top}}$. 

\begin{theorem}\cite{bs:2016} 
Let $\Gamma$ be a holomorphic correspondence on $\widehat{\mathbb{C}}$, as written in Definition \ref{holo} satisfying the condition $d_{{\rm top}} > d_{{\rm fwd}}$. Suppose $x_{0} \in \widehat{\mathbb{C}}\; \setminus\; \mathcal{E}$, where $\mathcal{E}$ is some polar set. Then, the sequence of measures $\displaystyle{\left\{ \dfrac{1}{\left( d_{{\rm top}} \right)^{n}} \sum\limits_{\begin{minipage}{3.6cm} \begin{center} \scriptsize{$x\, \in\, \widehat{\mathbb{C}}\; :$} \\ \scriptsize{$\Pi_{\left(n,\, n\right)}^{-} \left( \mathfrak{X}_{n}^{-} \left( x_{0}; \boldsymbol{\beta} \right)_{\boldsymbol{k}} \right)\ =\ x$} \end{center} \end{minipage}} \delta_{x} \right\}_{n\, \ge\, 1}}$ converges (in the weak* topology) to some measure, independent of $x_{0}$. 
\end{theorem} 

This limiting measure is called the \emph{Dinh-Sibony measure}, that we denote in this paper by $\omega_{{\rm DS}}$. The support of the measure $\omega_{{\rm DS}}$, denoted by $\Omega_{{\rm DS}}$ remains away from the normality set of $\Gamma$, as proved in \cite{bs:2016}. Further, $\Omega_{{\rm DS}}$ is a compact subset of $\widehat{\mathbb{C}}$ that is backward invariant under the action of $\Gamma$, {\it i.e.}, for any $x_{0} \in \Omega_{{\rm DS}}$, 
\begin{eqnarray*} 
\Omega_{{\rm DS}} & \supseteq & \bigcup_{n\, \ge\, 1} \left\{ x\, \in\, \widehat{\mathbb{C}}\; :\; \Pi_{\left(n,\, n\right)}^{-} \left( \mathfrak{X}_{n}^{-} \left( x_{0}; \boldsymbol{\beta} \right)_{\boldsymbol{k}} \right)\; =\; x \right\} \\ 
& = & \bigcup_{n\, \ge\, 1} \left\{ x\, \in\, \widehat{\mathbb{C}}\; :\; \Pi_{\left(n,\, n\right)}^{+} \left( \mathfrak{X}_{n}^{+} \left( x; \boldsymbol{\alpha} \right)_{\boldsymbol{j}} \right)\; =\; x_{0} \right\}. 
\end{eqnarray*} 

We intend to study the dynamics of $\Gamma$ restricted on $\Omega_{{\rm DS}}$; however, in order to ensure ``good behaviour" of forward iterates, we appeal to a result due to Londhe, that we now state. 

\begin{theorem}\cite{ml:2022} 
\label{mlt} 
Let $\Gamma$ be a holomorphic correspondence on $\widehat{\mathbb{C}}$ with $d_{{\rm top}} > d_{{\rm fwd}}$ and $\omega_{{\rm DS}}$ be its Dinh-Sibony measure, whose support is denoted by $\Omega_{{\rm DS}}$. If $x_{0} \in \Omega_{{\rm DS}}$, then there exists at least one $\boldsymbol{\alpha} \in \left\{ 1, 2, \cdots, M \right\}^{n}$ with corresponding value(s) of $\boldsymbol{j} \in \mathbb{Z}^{n}_{+}$ such that $\displaystyle{\Pi_{\left(n,\, n\right)}^{+} \left( \mathfrak{X}_{n}^{+} \left( x_{0}; \boldsymbol{\alpha} \right)_{\boldsymbol{j}} \right) \in \Omega_{{\rm DS}}}$, for all $n \in \mathbb{Z}_{+}$. 
\end{theorem} 

Now using Theorem \ref{mlt}, we define a subset of $\mathscr{P}^{\Gamma} \left( \Omega_{{\rm DS}} \right)$ wherein we collect the forward images of a point $x_{0} \in \Omega_{{\rm DS}}$ that also remain in $\Omega_{{\rm DS}}$, {\it i.e.}, define 
\[ \mathscr{P}^{\Gamma}_{{\rm inv.}} \left( \Omega_{{\rm DS}} \right)\ \ =\ \ \left\{ \left( x_{0}, x^{(1)}_{j_{1}}, x^{(2)}_{j_{2}}, \cdots;\; \alpha_{1}, \alpha_{2}, \cdots \right)\; \in\; \mathscr{P}^{\Gamma} \left( \Omega_{{\rm DS}} \right)\ :\ x^{(r)}_{j_{r}}\; \in\; \Omega_{{\rm DS}}\  \forall r \in \mathbb{Z}_{+} \right\}. \] 

Note that, by definition any $\displaystyle{\left( x^{(r - 1)}_{j_{r - 1}}, x^{(r)}_{j_{r}} \right) \in \Gamma\vert_{\Omega_{{\rm DS}} \times \Omega_{{\rm DS}}}}$ for all $r \in \mathbb{Z}_{+}$. Further, $\mathscr{P}^{\Gamma}_{{\rm inv.}} \left( \Omega_{{\rm DS}} \right)$ is an automatic metric space with respect to the appropriate restriction of the metric $\Delta$, defined in Equation \eqref{eqn:infmet}. Moreover, $\mathscr{P}^{\Gamma}_{{\rm inv.}} \left( \Omega_{{\rm DS}} \right)$ is a closed subset of $\mathscr{P}^{\Gamma} \left( \widehat{\mathbb{C}} \right)$, and hence, compact. Also note that the restriction of the shift map $\sigma^{\Gamma}$, as defined in Equation \eqref{sigma}, on $\mathscr{P}^{\Gamma}_{{\rm inv.}} \left( \Omega_{{\rm DS}} \right)$ yields the map to be forward invariant. Since we know $\Omega_{{\rm DS}}$ to be backward invariant under the action of the correspondence $\Gamma$, we essentially have $\sigma^{\Gamma}$ to be completely invariant on $\mathscr{P}^{\Gamma}_{{\rm inv.}} \left( \Omega_{{\rm DS}} \right)$ {\it i.e.}, $\sigma^{\Gamma} \left(\mathscr{P}^{\Gamma}_{{\rm inv.}} \left( \Omega_{{\rm DS}} \right) \right) = \mathscr{P}^{\Gamma}_{{\rm inv.}} \left( \Omega_{{\rm DS}} \right)$. 

\begin{definition} 
\label{expcorr} 
A holomorphic correspondence $\Gamma$ defined on $\widehat{\mathbb{C}}$, as written in Definition \ref{holo}, is said to be an \emph{expansive correspondence on $\Omega_{{\rm DS}}$} if there exists a constant $\lambda > 1$ such that for any pair of points $x_{0}, y_{0} \in \Omega_{{\rm DS}}$ with $\displaystyle{\mathfrak{X}_{1}^{-} \left( x_{0}; \beta \right)_{j} \in \mathscr{Q}_{1}^{\Gamma} \left( \widehat{\mathbb{C}} \right)}$, there exists a point $\displaystyle{\mathfrak{X}_{1}^{-} \left( y_{0}; \beta \right)_{k} \in \mathscr{Q}_{1}^{\Gamma} \left( \widehat{\mathbb{C}} \right)}$ such that 
\[ \lambda\ d_{\widehat{\mathbb{C}}} \left( \Pi_{\left(1,\, 1\right)}^{-} \left( \mathfrak{X}_{1}^{-} \left( x_{0}; \beta \right)_{j} \right),\; \Pi_{\left(1,\, 1\right)}^{-} \left( \mathfrak{X}_{1}^{-} \left( y_{0}; \beta \right)_{k} \right) \right)\ \ <\ \ d_{\widehat{\mathbb{C}}} \left( x_{0}, y_{0} \right). \] 
\end{definition} 

In this final section, we shall consider holomorphic correspondences $\Gamma$ defined on $\widehat{\mathbb{C}}$, as written in Definition \ref{holo}, whose restriction on the support, $\Omega_{{\rm DS}}$ of its Dinh-Sibony measure, $\omega_{{\rm DS}}$ is expansive. 

Let $\mathcal{C} \left( \Omega_{{\rm DS}}, \mathbb{R} \right)$ denote the Banach space of real-valued continuous functions defined on the compact metric space $\Omega_{{\rm DS}}$, equipped with the supremum norm. For any fixed $f \in \mathcal{C} \left( \Omega_{{\rm DS}}, \mathbb{R} \right)$, consider the bounded linear Ruelle operator on $\mathcal{C} \left( \Omega_{{\rm DS}}, \mathbb{R} \right)$, as defined in \cite{ss:2025}, denoted by $\mathcal{L}_{f}$ whose action on the points in $\Omega_{{\rm DS}}$ is as prescribed below. 
\begin{equation} 
\label{eqn:ruellef}
\left( \mathcal{L}_{f} (g) \right) (x)\ \ =\ \ \sum_{1\, \le\, \beta\, \le\, M} \sum\limits_{\begin{minipage}{3.25cm} \begin{center} \scriptsize{$1\, \le\, k\, \le\, \delta_{\beta} (x)$} \\ \scriptsize{$\Pi_{\left(1,\, 1\right)}^{-} \left( \mathfrak{X}_{1}^{-} \left( x; \beta \right)_{k} \right)\, =\, y$} \end{center} \end{minipage}} e^{f(y)} g(y). 
\end{equation} 

Let $\mathcal{L}_{f}^{*}$ denote the adjoint operator, corresponding to the Ruelle operator $\mathcal{L}_{f}$, defined on the dual space of $\mathcal{C} \left( \Omega_{{\rm DS}}, \mathbb{R} \right)$, namely the space of all signed measures supported on $\Omega_{{\rm DS}}$, denoted by $\mathcal{C} \left( \Omega_{{\rm DS}}, \mathbb{R} \right)^{*}$. For every signed measure $\rho \in \mathcal{C} \left( \Omega_{{\rm DS}}, \mathbb{R} \right)^{*}$, the action of the measure $\mathcal{L}_{f}^{*} (\rho)$ on any $g \in \mathcal{C} \left( \Omega_{{\rm DS}}, \mathbb{R} \right)$ is given by 
\[ \left[ \mathcal{L}_{f}^{*} (\rho) \right] (g)\ \ =\ \ \rho \left( \mathcal{L}_{f} (g) \right)\ \ =\ \ \int_{\Omega_{{\rm DS}}} \left( \mathcal{L}_{f} g \right) \mathrm{d} \rho. \] 
Observe that the set of probability measures supported on $\Omega_{{\rm DS}}$, denoted by $\mathscr{M} \left( \Omega_{{\rm DS}} \right)$ is a subset of $\mathcal{C} \left( \Omega_{{\rm DS}}, \mathbb{R} \right)^{*}$. For every $f \in \mathcal{C} \left( \Omega_{{\rm DS}}, \mathbb{R} \right)$ and $k \in \mathbb{Z}_{+}$, let 
\begin{equation} 
\label{omegak} 
\omega_{k} (f)\ \ =\ \ \sup \left\{ \left| f(x) - f(y) \right|\ :\ d_{\widehat{\mathbb{C}}} (x, y) \le \frac{1}{\lambda^{k - 1}} \right\} 
\end{equation} 
and consider the subset $\mathcal{C}^{\alpha} \left( \Omega_{{\rm DS}}, \mathbb{R} \right) \subset \mathcal{C} \left( \Omega_{{\rm DS}}, \mathbb{R} \right)$ where $\alpha = \lambda^{-1}$ given by  
\[ \mathcal{C}^{\alpha} \left( \Omega_{{\rm DS}}, \mathbb{R} \right)\ \ =\ \ \left\{ f \in \mathcal{C} \left( \Omega_{{\rm DS}}, \mathbb{R} \right) : \sum_{k\, \ge\, 1} \omega_{k} (f) + \| f \|_{\infty} < \infty \right\}. \] 
Then, it is a well-known fact that $\mathcal{C}^{\alpha} \left( \Omega_{{\rm DS}}, \mathbb{R} \right)$ is a Banach space with respect to the norm, 
\[ \| f \|_{\alpha}\ \ =\ \ \sum_{k\, \ge\, 1} \omega_{k} (f) + \| f \|_{\infty}. \] 

Note that $F = f \circ \Pi_{0} \in  \mathcal{C} \left( \mathscr{P}^{\Gamma}_{{\rm inv.}} \left( \Omega_{{\rm DS}} \right), \mathbb{R} \right)$, whenever $f \in \mathcal{C} \left( \Omega_{{\rm DS}}, \mathbb{R} \right)$. Suppose $\mathcal{L}_{F}$ denotes the classical Ruelle operator on $\mathcal{C} \left( \mathscr{P}^{\Gamma}_{{\rm inv.}} \left( \Omega_{{\rm DS}} \right), \mathbb{R} \right)$, where the dynamics in the underlying space happens through the shift map $\sigma^{\Gamma}$, then, for any $G \in \mathcal{C} \left( \mathscr{P}^{\Gamma}_{{\rm inv.}} \left( \Omega_{{\rm DS}} \right), \mathbb{R} \right)$, the action of the classical Ruelle operator is given by 
\begin{equation} 
\label{eqn:ruelleF}
\mathcal{L}_{F} (G) \left( \mathfrak{X}^{+} \left( x_{0}; \boldsymbol{\alpha} \right)_{\boldsymbol{j}} \right)\ \ = \hspace{-10pt} \sum\limits_{\begin{minipage}{5cm} \begin{center} \scriptsize{\begin{eqnarray*} \mathfrak{X}^{+} \left( x; \boldsymbol{\beta} \right)_{\boldsymbol{k}} & \hspace{-5pt} \in \hspace{-5pt} & \mathscr{P}^{\Gamma}_{{\rm inv.}} \left( \Omega_{{\rm DS}} \right) : \\ \sigma^{\Gamma} \left( \mathfrak{X}^{+} \left( x; \boldsymbol{\beta} \right)_{\boldsymbol{k}} \right) & \hspace{-5pt} = \hspace{-5pt} & \mathfrak{X}^{+} \left( x_{0}; \boldsymbol{\alpha} \right)_{\boldsymbol{j}} \end{eqnarray*}} \end{center} \end{minipage}} e^{F \left( \mathfrak{X}^{+} \left( x; \boldsymbol{\beta} \right)_{\boldsymbol{k}} \right)} G \left( \mathfrak{X}^{+} \left( x; \boldsymbol{\beta} \right)_{\boldsymbol{k}} \right). 
\end{equation} 
Then, it is clear that $\boldsymbol{k} = \left( j_{0}, j_{1}, j_{2}, \cdots \right) = j_{0} \boldsymbol{j}$ if $\boldsymbol{j} = \left( j_{1}, j_{2}, \cdots \right)$ and $\boldsymbol{\beta} = \left( \alpha_{0}, \alpha_{1}, \alpha_{2}, \cdots \right) = \alpha_{0} \boldsymbol{\alpha}$ if $\boldsymbol{\alpha} = \left( \alpha_{1}, \alpha_{2}, \cdots \right)$. Thus, 
\[ \mathcal{L}_{F} (G) \left( \mathfrak{X}^{+} \left( x_{0}; \boldsymbol{\alpha} \right)_{\boldsymbol{j}} \right)\ \ =\ \ \sum_{\alpha_{0}\, =\, 1}^{M} \sum\limits_{\begin{minipage}{4cm} \begin{center} \scriptsize{$1\, \le\, j_{0}\, \le\, \lambda_{\alpha_{0}} (x)$} : \\ \scriptsize{$\Pi_{1} \left( \mathfrak{X}^{+} \left( x; \alpha_{0} \boldsymbol{\alpha} \right)_{j_{0} \boldsymbol{j}} \right)\, =\, x_{0}$} \end{center} \end{minipage}} e^{F \left( \mathfrak{X}^{+} \left( x; \alpha_{0} \boldsymbol{\alpha} \right)_{j_{0} \boldsymbol{j}} \right)} G \left( \mathfrak{X}^{+} \left( x; \alpha_{0} \boldsymbol{\alpha} \right)_{j_{0} \boldsymbol{j}} \right). \] 

We now state and prove a theorem, concerning the existence of a unique measure in $\mathscr{S}^{\Gamma}$, corresponding to the adjoint of the Ruelle operator $\mathcal{L}_{f}$ for some $f \in \mathcal{C}^{\alpha} \left( \Omega_{{\rm DS}}, \mathbb{R} \right)$, as given in Equation \eqref{eqn:ruellef}, defined on $\mathcal{C} \left( \Omega_{{\rm DS}}, \mathbb{R} \right)$. 

\begin{theorem} 
\label{rotm} 
Let $\Gamma$ be a holomorphic correspondence on $\widehat{\mathbb{C}}$, as written in Definition \ref{holo}, for which $d_{{\rm top}} > d_{{\rm fwd}}$. Let $\omega_{{\rm DS}}$ be its Dinh-Sibony measure, whose support is denoted by $\Omega_{{\rm DS}}$. Assume that $\Gamma$ is expansive on $\Omega_{{\rm DS}}$. For the Ruelle operator $\mathcal{L}_{f}$ corresponding to a function $f \in \mathcal{C}^{\alpha} \left( \Omega_{{\rm DS}}, \mathbb{R} \right)$, as given in Equation \eqref{eqn:ruellef}, defined on $\mathcal{C} \left( \Omega_{{\rm DS}}, \mathbb{R} \right)$, there exists a unique measure $\nu \in \mathscr{S}^{\Gamma}$ such that for any $g \in \mathcal{C} \left( \Omega_{{\rm DS}}, \mathbb{R} \right)$, we have the sequence $\displaystyle{\left\{ \dfrac{1}{\Lambda^{n}} \mathcal{L}_{f}^{n} (g) \right\}_{n\, \ge\, 1} \to h \int \frac{g}{h} \mathrm{d}\nu}$ where $h \in \mathcal{C} \left( \Omega_{{\rm DS}}, \mathbb{R} \right)$ is the eigenfunction corresponding to the maximal simple positive eigenvalue $\Lambda$ of $\mathcal{L}_{f}$. 
\end{theorem} 

Before we begin the proof of this theorem, we state the following theorem due to the authors, that is an amalgamation of various results concerning the Ruelle operator, as one may find in \cite{ss:2025}. 

\begin{theorem} 
\label{rot} 
Let $\Gamma$ be a holomorphic correspondence on $\widehat{\mathbb{C}}$, as written in Definition \ref{holo}, for which $d_{{\rm top}} > d_{{\rm fwd}}$. Let $\omega_{{\rm DS}}$ be its Dinh-Sibony measure, whose support is denoted by $\Omega_{{\rm DS}}$. Assume that $\Gamma$ is expansive on $\Omega_{{\rm DS}}$. Consider the Ruelle operator $\mathcal{L}_{f}$ corresponding to a function $f \in \mathcal{C}^{\alpha} \left( \Omega_{{\rm DS}}, \mathbb{R} \right)$, as given in Equation \eqref{eqn:ruellef}, defined on $\mathcal{C} \left( \Omega_{{\rm DS}}, \mathbb{R} \right)$. For $F = f \circ \Pi_{0}$, consider the counterpart of $\mathcal{L}_{f}$ given by $\mathcal{L}_{F}$, as given in Equation \eqref{eqn:ruelleF}, defined on $\mathcal{C} \left( \mathscr{P}^{\Gamma}_{{\rm inv.}} \left( \Omega_{{\rm DS}} \right), \mathbb{R} \right)$. Then, the following statements hold. 
\begin{enumerate} 
\item There exists a simple, positive, maximal eigenvalue, say $\Lambda$ of the Ruelle operators $\mathcal{L}_{f}$ and $\mathcal{L}_{F}$ with corresponding positive eigenfunctions, say $h \in \mathcal{C} \left( \Omega_{{\rm DS}}, \mathbb{R} \right)$ and $H = h \circ \Pi_{0} \in \mathcal{C} \left( \mathscr{P}^{\Gamma}_{{\rm inv.}} \left( \Omega_{{\rm DS}} \right), \mathbb{R} \right)$ respectively. 
\item Suppose we define $\widetilde{F} = F - \log \left( H \circ \sigma^{\Gamma} \right) + \log H - \log \Lambda$, then the Ruelle operator pertaining to $\widetilde{F}$ is normalised, {\it i.e.}, it satisfies $\mathcal{L}_{\widetilde{F}} \mathbb{1} = \mathbb{1}$, where $\mathbb{1} \in \mathcal{C} \left( \mathscr{P}^{\Gamma}_{{\rm inv.}} \left( \Omega_{{\rm DS}} \right), \mathbb{R} \right)$ denotes the constant function evaluating to unity. 
\item For any $n \in \mathbb{Z}_{+}$, let $\mathcal{L}_{F}^{n}$ and $\mathcal{L}_{\widetilde{F}}^{n}$ denote the $n$-th iterate of the Ruelle operator $\mathcal{L}_{F}$ and the normalised Ruelle operator $\mathcal{L}_{\widetilde{F}}$ respectively. Then, the sequence $\displaystyle{\left\{ \mathcal{L}_{\widetilde{F}}^{n} G \right\}_{n\, \ge\, 1}}$ converges to a constant, say $c(G)$, for every continuous function $G \in \mathcal{C} \left( \mathscr{P}^{\Gamma}_{{\rm inv.}} \left( \Omega_{{\rm DS}} \right), \mathbb{R} \right)$. And thus, the sequence $\displaystyle{\left\{ \frac{1}{\Lambda^{n}} \mathcal{L}_{F}^{n} G \right\}_{n\, \ge\, 1}}$ converges to the function $\displaystyle{c \left( \frac{G}{H} \right) H}$. 
\item For any $n \in \mathbb{Z}_{+}$, if $\mathcal{L}_{f}^{n}$ denotes the $n$-th iterate of the Ruelle operator $\mathcal{L}_{f}$, then the sequence $\left\{ \dfrac{1}{\Lambda^{n}} \mathcal{L}_{f}^{n} g \right\}_{n\, \ge\, 1}$ converges uniformly in $\Omega_{{\rm DS}}$ to the function $c \left( \dfrac{g}{h} \circ \Pi_{0} \right) h$, for every $g \in \mathcal{C} \left( \Omega_{{\rm DS}}, \mathbb{R} \right)$. 
\end{enumerate} 
\end{theorem} 

We now make use of Theorem \ref{rot} and all the notations introduced therein to prove Theorem \ref{rotm}. 

\begin{proof}[of Theorem \ref{rotm}] 
Given $f \in \mathcal{C}^{\alpha} \left( \Omega_{{\rm DS}}, \mathbb{R} \right)$ as in the statement of Theorem \ref{rotm}, let $F = f \circ \Pi_{0} \in  \mathcal{C} \left( \mathscr{P}^{\Gamma}_{{\rm inv.}} \left( \Omega_{{\rm DS}} \right), \mathbb{R} \right)$ and $\widetilde{F} = F - \log \left( H \circ \sigma^{\Gamma} \right) + \log H - \log \Lambda$, pertaining to the normalisation of the appropriate Ruelle operator, as mentioned in statement (2) of Theorem \ref{rot}. Now, consider the action of the adjoint of the normalised Ruelle operator, denoted by $\mathcal{L}_{\widetilde{F}}^{*}$ defined on the space of probability measures supported on $\mathscr{P}^{\Gamma}_{{\rm inv.}} \left( \Omega_{{\rm DS}} \right)$, denoted by $\mathscr{M} \left( \mathscr{P}^{\Gamma}_{{\rm inv.}} \left( \Omega_{{\rm DS}} \right) \right)$ and given by $\displaystyle{\int G \mathrm{d} \left( \mathcal{L}_{\widetilde{F}}^{*} \mu \right) = \int \mathcal{L}_{\widetilde{F}} G \mathrm{d}\mu}$ for any $G \in \mathcal{C} \left( \mathscr{P}^{\Gamma}_{{\rm inv.}} \left( \Omega_{{\rm DS}} \right), \mathbb{R} \right)$. Then, for the constant function $\mathbb{1}$, we have $\displaystyle{\int \mathbb{1} \mathrm{d} \left( \mathcal{L}_{\widetilde{F}}^{*} \mu \right) = \int \mathcal{L}_{\widetilde{F}} \mathbb{1} \mathrm{d}\mu = \int \mathbb{1} \mathrm{d}\mu = 1}$, and thereby, $\mathcal{L}_{\widetilde{F}}^{*} \mu \in \mathscr{M} \left( \mathscr{P}^{\Gamma}_{{\rm inv.}} \left( \Omega_{{\rm DS}} \right) \right)$. 

Hence, one may consider $\mathcal{L}_{\widetilde{F}}^{*} : \mathscr{M} \left( \mathscr{P}^{\Gamma}_{{\rm inv.}} \left( \Omega_{{\rm DS}} \right) \right) \longrightarrow \mathscr{M} \left( \mathscr{P}^{\Gamma}_{{\rm inv.}} \left( \Omega_{{\rm DS}} \right) \right)$ as a function that maps $\mu \longmapsto \mathcal{L}_{\widetilde{F}}^{*} \mu$. An application of the Schauder-Tynchnoff theorem, as one may find in \cite{mz:1996}, then says that $\mathcal{L}_{\widetilde{F}}^{*}$ admits a fixed point, say $\mu_{0} \in \mathscr{M} \left( \mathscr{P}^{\Gamma}_{{\rm inv.}} \left( \Omega_{{\rm DS}} \right) \right)$. Then, for any $G \in \mathcal{C} \left( \mathscr{P}^{\Gamma}_{{\rm inv.}} \left( \Omega_{{\rm DS}} \right), \mathbb{R} \right)$, we have 
\begin{equation} 
\label{eqn:cg} 
\int G \mathrm{d}\mu_{0}\ \ =\ \ \int G \mathrm{d}\left( \lim_{n\, \to\, \infty} \left( \mathcal{L}_{\widetilde{F}}^{*} \right)^{n} \mu_{0} \right)\ \ =\ \ \int \lim_{n\, \to\, \infty} \left( \mathcal{L}_{\widetilde{F}}^{n} G \right) \mathrm{d}\mu_{0}\ \ =\ \ \int c(G) \mathrm{d}\mu_{0}\ \ =\ \ c(G), 
\end{equation} 
using statement (3) of Theorem \ref{rot} and thereby, making the fixed point $\mu_{0}$ of $\mathcal{L}_{\widetilde{F}}^{*}$ unique. Moreover, 
\[ \int \left( G \circ \sigma^{\Gamma} \right) \mathrm{d}\mu_{0}\ \ =\ \ \int \left( G \circ \sigma^{\Gamma} \right) \mathrm{d}\left(\mathcal{L}_{\widetilde{F}}^{*} \mu_{0} \right)\ \ =\ \ \int \left( \mathcal{L}_{\widetilde{F}} G \circ \sigma^{\Gamma} \right) \mathrm{d}\mu_{0}\ \ =\ \ \int G \mathrm{d}\mu_{0}, \] 
ensuring the membership of $\mu_{0}$ in the space of $\sigma^{\Gamma}$-invariant measures in $\mathscr{M} \left( \mathscr{P}^{\Gamma}_{{\rm inv.}} \left( \Omega_{{\rm DS}} \right) \right)$, denoted by $\mathscr{M}_{\sigma^{\Gamma}} \left( \mathscr{P}^{\Gamma}_{{\rm inv.}} \left( \Omega_{{\rm DS}} \right) \right)$. 

Finally, combining the statements (3) and (4) of Theorem \ref{rot} and using Equation \eqref{eqn:cg}, we have for any $g \in \mathcal{C} \left( \Omega_{{\rm DS}}, \mathbb{R} \right)$, 
\[ \frac{1}{\Lambda^{n}} \mathcal{L}_{f}^{n} g\ \ \to\ \ c \left( \frac{g}{h} \circ \Pi_{0} \right) h\ \ =\ \ h \int \left( \frac{g}{h} \circ \Pi_{0} \right) \mathrm{d}\mu_{0}\ \ =\ \ h \int \frac{g}{h} \mathrm{d}\left( \left( \Pi_{0} \right)_{*}\mu_{0} \right). \] 
Defining $\nu = \left( \Pi_{0} \right)_{*}\mu_{0} \in \mathscr{S}^{\Gamma}$ completes the proof. 
\end{proof}

\end{document}